\documentclass[a4paper,11pt,reqno]{amsart}
\usepackage[latin1]{inputenc}
\usepackage[T1]{fontenc}
\usepackage[english]{babel}
\def\uppercasenonmath#1{} 
\usepackage[margin={25mm,30mm}]{geometry}
\usepackage{calc}
\usepackage{pgfplots}
\usetikzlibrary{shapes.multipart}
\pgfplotsset{compat=1.17}
\usepackage{tkz-euclide}
\newcommand\circled[1]{\tikz[baseline=(char.base)]{%
\node[shape=circle,draw,inner sep=2pt] (char) {#1};}}
\usepackage{enumitem}
\usepackage{amsmath}  %
\usepackage{amsthm}   
\usepackage{amssymb}  
\usepackage{amsaddr}  
\numberwithin{equation}{section}
\theoremstyle{plain}
\newtheorem{theorem}{Theorem}[section]
\newtheorem*{theorem*}{Theorem}
\newtheorem{lemma}[theorem]{Lemma}
\newtheorem{proposition}[theorem]{Proposition}
\newtheorem{corollary}[theorem]{Corollary}
\theoremstyle{definition} 
\newtheorem{definition}[theorem]{Definition}
\newtheorem{remark}[theorem]{Remark}
\newtheorem{example}[theorem]{Example}
\usepackage{xcolor}
\newcommand{\bp}{\mathbf{p}}
\newcommand{\bq}{\mathbf{q}}
\newcommand{\cA}{\mathcal{A}} 
\newcommand{\cB}{\mathcal{B}} 
\newcommand{\cF}{\mathcal{F}} 
\newcommand{\cN}{\mathcal{N}} 
\newcommand{\dd}{\mathrm{d}} 
\newcommand{\N}{\mathbb{N}} 
\newcommand{\R}{\mathbb{R}} 
\DeclareMathOperator\aff{aff} 
\DeclareMathOperator\conv{co} 
\DeclareMathOperator\env{env} 
\DeclareMathOperator{\ext}{ext} 
\newcommand\ait{\mathrm{a.i.}}
\newcommand\extt{\mathrm{ext}}
\newcommand\fint{\mathrm{fin}}
\usepackage[colorlinks=true,urlcolor=black,citecolor=black,linkcolor=black]{hyperref}
\begin{document}
\title{Extreme points and faces in the moment problem}
\author[Didier Henrion \and Martin Kru\v{z}\'ik \and Stephan Weis]%
{Didier Henrion$^{\dagger,\ddagger}$, 
Martin Kru\v{z}\'ik$^{\S,\#}$, 
and Stephan Weis$^\dagger$}
\address{%
$^\dagger$Faculty of Electrical Engineering,
Czech Technical University in Prague, Czechia\\
$^\ddagger$LAAS-CNRS, 
University of Toulouse, France\\
$^\S$Institute of Information Theory and Automation,
Czech Academy of Sciences, Prague, Czechia\\
$^\#$Faculty of Civil Engineering, 
Czech Technical University in Prague, Czechia}
\email{henrion@laas.fr, 
kruzik@utia.cas.cz,
\href{mailto:maths@weis-stephan.de}{maths@weis-stephan.de}}
\thanks{CRediT: D.~Henrion and M.~Kru\v{z}\'{i}k contributed to 
Funding acquisition, Project administration, Conceptualization, Supervision, 
Writing - Review \& Editing;
S.~Weis is accountable for Methodology, Investigation, Formal Analysis, Visualization, 
and Writing - Original Draft}
\date{Draft of \today}
\subjclass[2020]{52A05, 46E27, 46N10, 44A60, 74B20}
\keywords{convex set, 
extreme point, 
affine constraints,
face,
smallest face, 
convex geometry of the probability simplex,
generalized moment problem,
generalized convex envelope,
polyconvex envelope}
%
%
%
%
%
%
\begin{abstract}
The polyconvex envelope, used in the calculus of variations and elasticity theory, was 
expressed by Dacorogna pointwise as a linear program on finitely atomic measures on the 
space of $m\times n$ matrices. Weizs\"acker and Winkler proved that the corresponding 
linear program on Borel measures restricts to the extreme points without increasing the 
infimum. Combining the two, one obtains a speed-up of grid-based algorithms and a new 
proof that the polyconvex envelope can be computed by the moment sum-of-squares 
hierarchy. Motivated by these applications, we seize the essence of extreme points in 
moment problems. First, we characterize extreme points of an affinely constrained convex 
set by the injectivity of the constraint map on the smallest faces containing them. We 
then study finitely many moment constraints. The extreme points are finitely atomic 
measures that have an affine independence property, under natural assumptions. We 
retrieve this known result with a simplified proof and apply it to faces of the probability 
simplex, among them the face of Radon measures. In the converse, we find that the 
assumption of a simplex is redundant. The Richter-Tchakaloff theorem allows us to show 
that the infimum of an integral functional restricts to the extreme points without 
increasing the infimum, not just for the known case of Radon measures but for any convex 
set of probability measures that contains the point measures.
\end{abstract}
\maketitle
%
%
%
\section{Introduction}
Many difficult nonlinear optimization problems can be approximated or solved 
in terms of a generalized moment problem 
\cite{henrion-etal20,klerk-laurent19,lasserre10,shapiro01}, by minimizing an 
integral functional on an infinite-dimensional convex set $Q$ of measures under 
finitely many moment constraints. The corresponding feasibility problem is the 
classical moment problem \cite{schmuedgen17}.
\par
Owing to Weizs\"acker and Winkler's integral representation of Radon measures 
\cite{weizsaecker-winkler79}, the generalized moment problem on the simplex $Q=R$ 
of Radon probability measures restricts to the extreme points of the constraint 
set without increasing the infimum. The extreme points are finitely atomic 
measures that have an affine independence property \cite{weizsaecker-winkler80,winkler88}. 
They are concentrated on finite sets of sizes at most $d+1$ if $d<\infty$ many 
moment constraints are in place. 
\par
In this paper, we analyze extreme points in the generalized moment problem using 
facial structures. By definition, a face of a convex set is a convex subset that 
absorbs every closed segment from the convex set whose interior it intersects. An 
extreme point is a one-point face, that is to say, it does not lie in the interior 
of any segment in the convex set. Faces of the convex set $P$ of probability
measures have been studied before by Dubins \cite{dubins62}, Winkler 
\cite{winkler88,winkler2000}, Aliprantis and Border \cite{aliprantis-border06},
and more recently in  \cite{weis25,weis-shirokov21} and \cite{dufour-prieto-rumeau25}.
We exploit that $P$ is a simplex in the sense that the convex cone it generates is 
a lattice, as Yosida and Hewitt \cite{yosida-hewitt52} have shown. 
\par
We begin with affine constraints on an arbitrary convex set $K$ in the 
Sections~\ref{sec:main-result} and~\ref{sec:finite}. We find that the extreme points 
of the constraint set are distinguished by the injectivity of the constraint map on 
the smallest faces of $K$ containing them (Theorem~\ref{thm:ext}). These minimal 
faces have been studied before by Dubins \cite{dubins62} and Alfsen \cite{alfsen71},  
and more recently in the works 
\cite{garcia-pacheco20,gorokhovik25,millan-roshchina23,weis25,weis-shirokov21}. 
In Section~\ref{sec:finite} we retrieve the known characterization of extreme points 
under finitely many affine constraints \cite{weizsaecker-winkler80,winkler88}. Notably, 
we use the aforementioned injectivity criterion to simplify the proof that the extreme 
points are convex combinations of extreme points of $K$ that have an affine 
independence property.
\par
Section~\ref{sec:problem-of-moments} is concerned with moment constraints defined by
finitely many integral functionals on a convex subset $Q$ of the probability simplex $P$, 
for an arbitrary measurable space.
\par
Section~\ref{sec:ex-moment} improves known results on extreme points 
\cite{weizsaecker-winkler80,winkler88}. We show that the assumption that $Q$ is a simplex 
is redundant in the sufficient condition for points to be extreme points. In the converse,
necessary condition, that extreme points are finitely atomic measures that have an affine
independence property, we explore the assumption that $Q$ is a face of $P$ and we provide 
a criterion to identify faces. In particular, the above necessary condition is true for 
the faces $R$ of Radon and $P_\tau$ of $\tau$-smooth measures of the simplex of Borel 
probability measures. 
\par
Section~\ref{sec:measure-lp} is devoted to linear optimization. We discuss Weizs\"acker 
and Winkler's proof that restricting the generalized moment problem to the extreme points 
of the constraint set preserves the infimum if $Q=R$. The analogue is true more generally 
for any convex subset $Q\subset P$ that contains all point measures. Our proof of this 
new result employs the Richter-Tchakaloff theorem (consistent moment constraints admit a 
finitely atomic solution) \cite{schmuedgen17,shapiro01}, Bauer's theorem \cite{bauer58} 
in its simplest setting (every linear functional on a finite-dimensional compact set 
attains a minimum at an extreme point), and facial geometry. The characterization of 
extreme points plays an active role in our proof, whereas Weizs\"{a}cker and Winkler's 
proof for $Q=R$ is independent of it.
\par
We shed light on generalized convex envelopes in Section~\ref{sec:Dacorogna-MP}, 
proving they are optimal value functions of generalized moment problems. An 
example is the polyconvex envelope in nonlinear elasticity theory 
\cite{ball76,ciarlet21,kruzik-roubicek19}, which is the optimal value function of 
various generalized moment problems under affine constraints on the minors of a 
matrix. The polyconvex envelope matches the optimal value on finitely atomic 
measures \cite{dacorogna08} and on Borel measures, see \cite{fantuzzi-etal26} and 
the references therein. We show that it also matches the optimal value on the 
extreme points and on the finitely atomic measures that have the affine 
independence property.
\par
The concluding Section~\ref{sec:conclusion} mentions numerical approaches to the
polyconvex envelope: grid-based methods \cite{bartels05} and their speed-up enabled 
by the affine independence property, as well as the mesh-free method 
of the moment sum-of-squares hierarchy in polynomial optimization 
\cite{fantuzzi-etal26}. We also mention extreme points under infinitely many moment 
constraints, and, as an example beyond simplices, extreme points of 
energy-constrained quantum states.
\par
\subsection*{Contribution}
\begin{itemize}
\item 
proving a new injectivity criterion that characterizes extreme points under affine 
constraints on convex sets (Theorem~\ref{thm:ext})
\item 
simplifying the proof of the known characterization of extreme points for finitely 
many affine constraints (Remark~\ref{rem:our-contribution-finite})
\item 
removing the redundant assumption of a simplex from the known sufficient condition 
for points to be extreme points under moment constraints, applying the necessary 
condition to faces of the probability simplex, and characterizing such faces
(Remark~\ref{rem:ext-moment-contribution})
\item 
proving that the infimum is preserved when the generalized moment problem is 
restricted to the extreme points, thereby generalizing the known case of Radon measures to 
arbitrary convex sets of probability measures that contain the point measures 
(Remark~\ref{rem:our-contrib-opti})
\item 
proving that generalized convex envelopes are optimal value functions of generalized moment 
problems, obtaining new representations for the polyconvex envelope 
(Remark~\ref{rem:envelopes})
\end{itemize}   
%
%
\section{Extreme points under affine constraints}
\label{sec:main-result}
We present a novel characterization of extreme points under affine constraints 
using faces of convex sets. We provide self-contained proofs, and we review  
terminology and related works.
\par
Let $K$ and $C$ be convex sets in real vector spaces $V$ and $W$\!, respectively.
Let $\alpha\colon K\to W$ be an affine map and  
\begin{equation}\label{eq:constraint-set}
H:=\alpha^{-1}(C)
=\{x\in K : \alpha(x)\in C\}.
\end{equation}
The map $\alpha$ can be seen as a \emph{constraint} or \emph{constraint function} 
and the preimage $H$ as a \emph{feasible set} or \emph{constraint set} in 
optimization theory \cite{ben-tal-nemirovski01,jahn07}.
\par
An \emph{extreme set} of $K$ is a subset of $K$ that contains 
any two points $x,y\in K$ whenever it contains their convex combination 
$(1-\lambda)x + \lambda y$ for some $0<\lambda< 1$. A point 
$x\in K$ is an \emph{extreme point} of $K$ if $\{x\}$ is an extreme 
set of $K$. We denote the set of extreme points of $K$ by $\ext K$. 
A \emph{face} of $K$ is a convex extreme set of $K$. By $F_K(x)$ we denote the 
intersection of all faces of $K$ that contain a given point $x\in K$. 
\par
\begin{remark}\label{rem:extreme}
We adhere to Fremlin and Pryce' definition of an extreme set of a 
convex set \cite{fremlin-pryce74}, see also \cite{holmes75}. This
notion is used in Krein-Milman's theorem \cite{rudin91} and a similar 
concept appears in Bauer's maximum principle \cite{bauer58}, see 
\cite[pp.~110--113]{phelps01}. The notion of a face is rather standard
\cite{alfsen71,alfsen-shultz01,fremlin-pryce74,phelps01,rockafellar70,schneider14}.
The set $F_K(x)$ is the smallest face of $K$ that contains a given point $x\in K$ 
(in the partial order of inclusion) studied in \cite{dubins62} and 
\cite[Section~II.5]{alfsen71}, see also 
\cite{garcia-pacheco20,gorokhovik25,millan-roshchina23,weis25,weis-shirokov21}.
The definition of a face in \cite{dubins62} differs from the one presented here, 
see \cite[Section~8]{weis25} for details.
\end{remark}
Proposition~\ref{pro:Alfsen} is quoted from \cite[p.~121]{alfsen71}. 
Its proof is taken from \cite[Proposition~4.1]{weis25}. Let
\[
S_K(x):=\left\{ y\in K \mid \exists\epsilon>0\colon x+\epsilon(x-y)\in K \right\},
\quad x\in K.
\]
\par
\begin{proposition}[Alfsen]\label{pro:Alfsen}
For every $x\in K$ we have $F_K(x)=S_K(x)$.
\end{proposition}
\begin{proof}
It suffices to prove that $S_K(x)$ is a face of $K$, as $F_K(x)\supset S_K(x)$ 
is implied by $F_K(x)$ being an extreme set of $K$ containing $x$. Let 
$a,b,y\in K$ and $\lambda\in(0,1)$ satisfy $y=(1-\lambda)a+\lambda b$.
\par
\begin{itemize}
\item 
$S_K(x)$ is convex. If $a,b\in S_K(x)$, then there is $\epsilon>0$ such that 
$a':=x+\epsilon(x-a)$ and $b':=x+\epsilon(x-b)$ lie in $K$. Then 
$(1-\lambda)a'+\lambda b'=x+\epsilon(x-y)$ shows $y\in S_K(x)$.
\item
$S_K(x)$ is an extreme set. If $y\in S_K(x)$, then there is $\epsilon>0$ 
such that $y':=x+\epsilon(x-y)$ lies in $K$. Let
$\lambda_a:=\epsilon \lambda/(1+\epsilon \lambda)$ and 
$\epsilon_a:=\epsilon(1-\lambda)/(1+\epsilon \lambda)$. Then 
$(1-\lambda_a)y'+\lambda_a b=x+\epsilon_a(x-a)$ shows $a\in S_K(x)$. 
Similarly $b\in S_K(x)$. 
\qedhere
\end{itemize}
\end{proof}
The scalars $\lambda_a,\epsilon_a$ in Proposition~\ref{pro:Alfsen} can be 
obtained from Menelaus' theorem in plane geometry, which asserts that a 
straight line divides the sides of a triangle in such a way that the 
product of the partial ratios is one \cite{gruenbaum-shephard95}.
\par
\begin{definition}[Dubins]\label{def:internal-point}
We call a point $x\in K$ a \emph{D-internal point} of $K$ if for every $y\in K$ 
there is $\epsilon>0$ such that $x+\epsilon(x-y)\in K$. 
\end{definition}
Let $\aff A$ be the \emph{affine hull} of a set $A\subset V$. This is the smallest 
affine space including $A$.
\par
\begin{remark}\label{remark:internal-points}
Definition~\ref{def:internal-point} replicates Dubins' concept of an ``internal 
point'' \cite{dubins62}, which we distinguish with the prefix ``D-'' here. 
Bourbaki's \emph{internal point} with respect to an affine space $A\supset K$
refers to a point $x\in K$ such that for every $a\neq x$ in $A$ there exists  
$\epsilon>0$ for which the open line segment 
$\{(1-\lambda)x+\lambda a: |\lambda|<\epsilon\}$ is included in $K$, see 
Section II.4.2 in \cite{bourbaki87}. 
\begin{itemize}
\item    
Internal points with respect to $V$ are simply called \emph{internal points} 
\cite{aliprantis-border06,bourbaki87,dunford-schwartz88,royden-fitzpatrick10}. 
They are the points $x\in K$ for which $K-x$ is \emph{absorbing} 
\cite{aliprantis-border06,bourbaki87,rudin91}. The set of internal points of 
$K$ is the \emph{core} or \emph{algebraic interior} 
\cite{barvinok02,holmes75,khazayel-etal21,zalinescu99} of $K$. 
\item 
The set of internal points of $K$ with respect to $\aff K$ is the 
\emph{intrinsic core} \cite{holmes75,khazayel-etal21,zalinescu99}
or \emph{relative algebraic interior} 
\cite{khazayel-etal21,weis-shirokov21,weis25,zalinescu99} of $K$. Clearly, the
intrinsic core of $K$ is included in the set of D-internal points  of $K$. The 
converse inclusion has been proven by Mill{\'a}n and Roshchina in 
\cite{millan-roshchina23}, Proposition 3.10, and independently in \cite{weis25}, 
Proposition~4.4. 
\end{itemize}
It is well known that the core%
\footnote{The intrinsic core of $K$ is the interior of $K$ in the relative topology of 
$\aff K$ induced by the finest locally convex topology on $V$\!, assuming Zermelo-Fraenkel 
set theory and the axiom of choice \cite[Section~3]{weis25}. If $\dim(K)<\infty$ then 
the intrinsic core of $K$ is known as the relative interior of $K$, which is nonempty if 
$K$ is so \cite[Section~6]{rockafellar70}.}
of $K$ is the interior of $K$ in the finest locally convex topology \cite[II.26]{bourbaki87} 
on $V$. The core of $K$ can be empty \cite{barvinok02,holmes75} even though $\aff(K)=V$.
\end{remark}
\begin{corollary}\label{cor:internal}
Every point $x\in K$ is a D-internal point of $F_K(x)$.
\end{corollary}
\begin{proof}
Proposition~\ref{pro:Alfsen} shows that for every $y\in F_K(x)$ there is $\epsilon>0$ such 
that $y':=x+\epsilon(x-y)$ lies in $K$. Since $F_K(x)$ is an extreme set of $K$ containing 
$x=\frac{\epsilon}{1+\epsilon}y+\frac{1}{1+\epsilon}y'$, it contains $y'$.
\end{proof}
\begin{figure}
\hfill
\begin{tikzpicture}
\draw[white] (-3.5,0) rectangle (1.5,3);
\tkzDefPoint(0,0){z}
\tkzDefPoint(-3,3){y}
\tkzDefPoint(3/2,3){yp}
\tkzDefPoint(-8/5,8/5){x}
\tkzDefPoint(4/5,8/5){xp}
\tkzDefPoint(-7/2,8/5){xpp}
\foreach \n in {z,x,xp,xpp,y,yp}
 \fill (\n) circle (2pt);
\tkzLabelPoint[right](z){$z$}
\tkzLabelPoint[above,xshift=0.7mm](x){$x$}
\tkzLabelPoint[right](xp){$x'$}
\tkzLabelPoint[left](xpp){$x''$}
\tkzLabelPoint[left](y){$y$}
\tkzLabelPoint[right](yp){$y'$}
\draw[thin,dotted,-,shorten >=2mm,shorten <=2mm] (xp) -- (yp);
\draw[thin,dotted,-,shorten >=2mm,shorten <=2mm] (yp) -- (y);
\draw[thick,-latex,shorten >=2mm,shorten <=2mm] 
 (y) -- (z) node[near end,below left] {\circled{1}};
\draw[thick,-latex,shorten >=2mm,shorten <=2mm] 
 (z) -- (xp) node[midway,below right] {\circled{2}};
\draw[thick,-latex,shorten >=2mm,shorten <=2mm] 
 (xp) -- (xpp) node[near start,above] {\circled{3}};
\end{tikzpicture}
\hspace*{\fill}
\caption{\label{fig:thm-char-ext}%
Sketch of the proof of injectivity for extreme points, by contraposition.
Let $x$ belong to the constraint set $H$ and let $y\neq y'$ in $F_K(x)$
satisfy $\alpha(y)=\alpha(y')$. In three steps 
\protect\circled{1}, \protect\circled{2}, and \protect\circled{3}, we
construct $z,x',x''$ in $F_K(x)$ such that $x''\neq x\neq x'$ and 
$\alpha(x'')=\alpha(x)=\alpha(x')$. Hence, $x$ is not an extreme point 
of $H$.}
\end{figure}
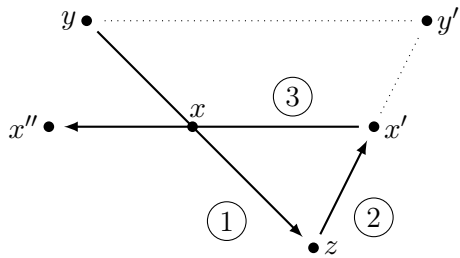
\begin{theorem}[Injectivity criterion]\label{thm:ext}~
\begin{enumerate}
\item 
If $x\in\ext H$ then $\alpha|_{F_K(x)}$ is injective.
\item 
If $C$ is a singleton, $x\in H$, and $\alpha|_{F_K(x)}$ is injective, 
then $x\in\ext H$. 
\end{enumerate}
\end{theorem}
\begin{proof}
(1) Proceeding indirectly, we assume that $\alpha|_{F_K(x)}$ fails to be 
injective. Let $y,y'\in F_K(x)$ such that $y'\neq y$ and 
$\alpha(y')=\alpha(y)$. As depicted in Figure~\ref{fig:thm-char-ext}, there are 
$\epsilon,\epsilon'>0$ such that $F_K(x)$ contains the points 
$z:=x+\epsilon(x-y)$,
\[\textstyle
x':=\frac{1}{1+\epsilon}(z+\epsilon y')
=x+\frac{\epsilon}{1+\epsilon}(y'-y),
\quad\mbox{and}\quad
x'':=x+\epsilon'(x-x'),
\]
because $x$ is a D-internal point of $F_K(x)$ by Corollary~\ref{cor:internal}.
The constraint set $H$ contains $x'$ and $x''$ since 
$\alpha(x'')=\alpha(x')=\alpha(x)\in C$. The claim $x\not\in\ext H$ then 
follows from $x'\neq x$.
\par
(2) Let $x=(1-\lambda)x'+\lambda x''$ for some $x',x''\in H$ and $\lambda\in(0,1)$. 
Since $F_K(x)$ is an extreme set of $K$ containing $x$, we have $x',x''\in F_K(x)$. 
As $C$ is a singleton, $\alpha(x'')=\alpha(x')$ holds, hence $x''=x'$ follows from 
the injectivity of $\alpha|_{F_K(x)}$, which proves $x\in\ext H$.
\end{proof}
\begin{remark}\label{rem:x-internal-gen-x}
We believe the facial analysis of infinite-dimensional convex sets benefits a lot from
Alfsen's Proposition~\ref{pro:Alfsen} and its Corollary~\ref{cor:internal}, whose simple 
and intuitive use is demonstrated in \cite{weis25} and in the present paper, especially
Theorem~\ref{thm:ext}. Corollary~\ref{cor:internal} has already been stated in 
\cite[Proposition~2.3 (5)]{garcia-pacheco20} and used in 
\cite[Corollary~3.6]{millan-roshchina23}, without proofs. 
Corollary~\ref{cor:internal} and the equivalent, \emph{a priori} stronger assertion that 
every $x\in K$ lies in the relative algebraic interior of $F_K(x)$ are proved in 
\cite[Section~4]{weis25}, and, in a more complicated way, in 
\cite[Theorem 2.3]{weis-shirokov21}.
\end{remark}
%
%
\section{Finitely many affine constraints}
\label{sec:finite}
We retrieve the known characterization \cite{weizsaecker-winkler80,winkler88}
of extreme points under finitely many affine constraints with a simpler proof. 
Unless stated otherwise $d:=\dim W$ is finite. 
\par
We begin with two classical theorems of convex geometry quoting \cite{schneider14}, 
Theorem~1.1.4 and Corollary~1.4.5. Let $\conv A$ be the \emph{convex hull} of a set 
$A\subset V$. This is the smallest convex set including $A$. By definition, 
$x_0,\dots,x_n\in V$ are \emph{affinely independent} if $x_1-x_0,\dots,x_n-x_0$ 
are linearly independent. 
\par
\begin{theorem*}[Carath{\'e}odory]
Let $A\subset V$ be a set whose affine hull has finite dimension $k$. 
Then every point in $\conv A$ is a convex combination of $k+1$ or fewer 
affinely independent points in $A$.
\end{theorem*}
\begin{theorem*}[Minkowski]
If $\dim K<\infty$ and $K$ is compact, then $K=\conv\ext K$.
\end{theorem*}
\begin{definition}
Let $\N:=\{1,2,3,\dots\}$ and $\N_0:=\{0\}\cup\N$. 
The convex set $K$ is 
\begin{enumerate}
\item 
\emph{linearly compact} 
\cite{kendall62,klee63-compact} if $\ell\cap K$ is a compact line 
segment (possibly degenerated to a singleton) for every line 
$\ell$ in $V$ that meets $K$.
\item 
\emph{$k$-neighborly}\footnote{%
The term \emph{$k$-neighborly} extends its homonym in the theory of 
polytopes \cite{gruenbaum03,sturmfels88,ziegler95}.}
if the convex hull of any $k$ or less extreme points of 
$K$ is a face of $K$, $k\in\N$.
\end{enumerate}
\end{definition}
\begin{proposition}[Weizs\"{a}cker-Winkler]\label{pro:ext-fin}
Let 
\begin{align}
\label{eq:Hprime1}
H' & \textstyle 
:= \{ x\in H : 
\text{$x$ is a convex combination of at most $d+1$ extreme points}\\
\nonumber
& \textstyle \hspace{\widthof{$:=\{$}}
\text{$x_1,\dots,x_k$ of $K$ and $\alpha(x_1),\dots,\alpha(x_k)$ 
are affinely independent\,}\}.
\end{align}
\begin{enumerate}
\item
If $K$ is linearly compact then $\ext H\subset H'$.
\item 
If $K$ is $(d+1)$-neighborly and if $C$ is a singleton,
then $H'\subset\ext H$.
\end{enumerate}
\end{proposition}
\begin{proof}
(1) Let $x\in\ext H$ and $F:=F_K(x)$. Then $\dim F\leq d$ as $\alpha|_F$ is injective by 
Theorem~\ref{thm:ext}~(1). Since $F$ is a face of $K$, we have $F=K\cap\aff F$. As $K$ is 
linearly compact and $\dim\aff F<\infty$, this implies that $F$ is compact
\cite[Section~4]{kendall62}. Thus, $x\in\conv\ext F$ by Minkowski's theorem, so $x$ is 
a convex combination of at most $d+1$ affinely independent extreme points $x_1,\dots,x_k$ 
of $F$ by Carath{\'e}odory's theorem. The points $x_1,\dots,x_k$ lie in $\ext K$, as $F$ is 
a face of $K$. Since $\alpha|_F$ is injective, the points $\alpha(x_1),\dots,\alpha(x_k)$ 
are affinely independent.
\par
(2) Let $x\in H'$. Then $x$ is a convex combination of at most $d+1$ extreme points 
$x_1,\dots,x_k$ of $K$. The affine independence of $\alpha(x_1),\dots,\alpha(x_k)$ implies 
that $\alpha$ restricted to $F:=\conv\{x_1,\dots,x_k\}$ is injective. Since $F$ is a face 
of the $(d+1)$-neighborly set $K$ we have $F_K(x)\subset F$. Thus $\alpha$ restricted to 
$F_K(x)$ is injective, so Theorem~\ref{thm:ext} (2) proves the claim.
\end{proof}
The affine independence property of the points in $H'$ is a recurring pattern
in this paper.
\par
\begin{example}[Counterexamples to Proposition~\ref{pro:ext-fin} 
under weaker assumptions]\label{exa:counter}~
\begin{enumerate}
\item 
The inclusion $\ext H\subset H'$ can fail if $K$ is not linearly compact.
An example is the open unit segment $K:=(0,1)$, the embedding 
$\alpha\colon K\to\R$, and $C:=\{\frac{1}{2}\}$.
\item 
All assumptions of Proposition~\ref{pro:ext-fin}~(2) are necessary 
for the inclusion $H'\subset\ext H$.
\begin{enumerate}
\item 
If $K\subset\R^2$ is a quadrilateral, $\alpha\colon K\to\R$, $(x,y)\mapsto x$,
and $C:=\{x_0\}$ for some $x_0\in\alpha(K)$, then $H=\alpha^{-1}(\{x_0\})$ is 
a segment (possibly degenerated to a point) and $H'$ is the intersection of 
$H$ and the union of all non-vertical edges and diagonals of $K$. Thus, $H'$ 
is typically a set of three or four mutually distinct points of which only two 
are in $\ext H$. (Quadrilaterals are not two-neighborly.)
\item
If $K$ is a triangle in the plane $\R^2$ and $\alpha\colon K\to\R$, $(x,y)\mapsto x$,
then, again, $H'$ is the intersection of $H$ and the union of the non-vertical edges 
of $K$. Unlike a quadrilateral, the triangle $K$ is two-neighborly. Still, 
$H'\subset\ext H$ fails whenever $C\cap\alpha(K)$ is more than just a point
(the assumption that $C$ is a singleton is violated).
\item 
Let $K$ be a compact convex set in $\R^d$ and let $\alpha\colon K\to\R^d$ be the 
embedding. Minkowski's and Carath\'{e}odory's theorems show $H'=H=C\cap K$. Thus, 
$H'\subset\ext H$ fails unless $C\cap K$ in a point or empty.
\end{enumerate}
\end{enumerate}
\end{example}
\begin{remark}[Background]\label{rem:background-WWext}
We collect conditions for a point $x\in H$ to be an extreme point of $H$. 
For every $k\in\N$, let $H''(k)$ be the set of all points in $H$ that  
are convex combinations of at most $k$ extreme points of $K$. Dubins 
\cite{dubins62} establishes the necessity of $x\in H''(d+1)$ for linearly 
compact $K$. More generally, if the smallest face $F_H(x)$ of 
$H$ containing $x$ has dimension $k\in\N_0$, then Klee \cite{klee63-dubins} 
shows $x\in H''(d+k+1)$. Weizs\"{a}cker and Winkler recognize the meaning 
of affine independence \cite[lemma on p.~136]{weizsaecker-winkler80} and 
\cite[Proposition~2.1]{winkler88}. They prove that $x\in H'$ is necessary 
for linearly compact $K$ and sufficient for singletons $C$ and 
$(d+1)$-neighborly $K$. 
\end{remark}
\begin{definition}[Simplices]\label{def:simplex}
The convex set $K$ is a \emph{simplex}%
\footnote{Simplices are called \emph{Choquet simplices} in \cite{winkler85}.}
if the convex cone 
\[
\{\lambda(x,1): \lambda\geq 0, x\in K\}\subset V\times\R
\] 
is a lattice in its own order \cite[Section~10]{phelps01}. A \emph{$k$-simplex}, $k\in\N_0$,
is defined as the convex hull of any $k+1$ affinely independent points. 
\end{definition}
It is well known that a $k$-dimensional convex set, $k\in\N_0$, is a simplex if and only if
it is a $k$-simplex \cite[Proposition~10.10]{phelps01}.
\par
\begin{remark}[Summary]\label{rem:our-contribution-finite}
The statement of Proposition~\ref{pro:ext-fin}~(1) that every extreme point $x$ of $H$ lies 
in $H'$, if $K$ is linearly compact, is originally proved using a family of $k$-simplices 
\cite{weizsaecker-winkler80,winkler88}. Our proof is shorter and, we believe, more intuitive: 
we apply Minkowski's and Carath{\'e}odory's theorems to a single convex set, the smallest 
face of $K$ containing $x$. 
\end{remark}
We quote arguments of convex geometry, to be used frequently.
\par
\begin{lemma}\label{lem:basic-convexity}~
\begin{enumerate}
\item 
The convex hull of every set of extreme points of a simplex is a face of 
this simplex. In particular, every simplex is $k$-neighborly for every 
$k\in\N$. 
\item 
Every simplex is linearly compact.
\item
Every face $F$ of $K$ inherits certain properties from $K$:
\begin{enumerate}
\item 
If $K$ is a simplex then so is $F$.
\item 
If $K$ is $k$-neighborly for some $k\in\N$ then so is $F$.
\item 
If $K$ is linearly compact then so is $F$.
\end{enumerate}
\end{enumerate}
\end{lemma}
\begin{proof}
(1) See \cite[Lemma~2.3]{winkler88}.
(2) See part (2\textsuperscript{o}) of the theorem in \cite{kendall62}. 
(3a) See \cite[Lemma~3.2]{winkler2000}; see also \cite[Lemma~10.4]{phelps01}. 
(3b) and (3c) are easy to prove.
\end{proof}
%
%
\section{Extreme points in generalized moment problems}
\label{sec:problem-of-moments}
We revisit generalized moment problems under finitely many moment constraints.
Section~\ref{sec:ex-moment} studies the extreme points of the constraint set
and Section~\ref{sec:measure-lp} is concerned with optimal values of integral 
functionals. Section~\ref{sec:Dacorogna-MP} associates generalized convex 
envelopes and optimal value functions.
\par
%
%
\subsection{Extreme points under moment constraints}
\label{sec:ex-moment}
We examine the extreme points of a convex set defined by finitely many moment 
constraints on a convex subset $Q$ of the probability simplex $P$. If $Q$ itself
is a simplex whose extreme points are point measures, then the extreme points are 
finitely atomic measures having the affine independence property 
\cite{weizsaecker-winkler80,winkler88}. We study the new assumption that $Q$ is 
a face of $P$ in this necessary condition for a point to be an extreme point. We 
prove that $Q$ being a simplex is redundant in the converse, sufficient condition.  
\par
\begin{definition}
Let $(X,\cF)$ be a measurable space whose $\sigma$-algebra $\cF$ contains all 
singletons in $X$. A \emph{measure} is a countably additive map $\cF\to[0,\infty)$. 
A measure $\mu$ is \emph{concentrated} on $A\in\cF$ if $\mu(X\setminus A)=0$.
A measure $\mu$ is a \emph{probability measure} if its total mass is $\mu(X)=1$. 
Let $P$ be the convex set of probability measures on $(X,\cF)$. The 
\emph{point measure} at $x\in X$, denoted $\delta_x$, is the measure with values 
$\delta_x(A):=1$ if $x\in A$ and $\delta_x(A):=0$ otherwise for every $A\in\cF$. 
Let 
\[
P_1:=\{\delta_x: x\in X\}.
\]
A measure $\mu$ is \emph{$k$-atomic} \cite{klerk-laurent19,lasserre10,nie23,schmuedgen17}, 
$k\in\N_0$, if $\mu=\lambda_1\delta_{x_1}+\dots+\lambda_k\delta_{x_k}$ for mutually 
distinct points $x_1,\dots,x_k$ in $X$ and strictly positive real weights 
$\lambda_1,\dots,\lambda_k>0$. A measure is \emph{finitely atomic} if it is $k$-atomic 
for some $k\in\N_0$. Let
\[
P\!_\fint:=
\{\mu\in P : \text{$\mu$ is finitely atomic}\,\}.
\]
\end{definition}
As $\cF$ contains all singletons, every finitely atomic measures is concentrated on a 
finite set each of whose points has strictly positive measure.
\par
\begin{remark}\label{rem:measures-basics}
We recall basic convex geometry of the convex set $P$ of probability measures.
\begin{enumerate}
\item
A classical result by Yosida and Hewitt states that $P$ is a simplex 
\cite[Theorem~1.14]{yosida-hewitt52}, see also \cite[Section~10.11]{aliprantis-border06}. 
Thus, we refer to $P$ as the \emph{probability simplex}. There are various nonequivalent 
definitions of simplices in infinite dimensions \cite[Section~1.5]{winkler85}. The one 
using barycenters of boundary measures is much more subtle than the order theoretic one 
(of Definition~\ref{def:simplex}) we are using in this paper.
\item 
It is well known that a measure $\mu$ on $\cF$ is an extreme point of $P$ if and 
only if $\mu$ is two-valued. This is easy to prove using conditional probability 
measures. In particular, $P_1\subset\ext P$ holds. Thereby, $\mu$ is \emph{two-valued} 
if $\mu(A)\in\{0,1\}$ for every $A\in\cF$. 
\item
The set $P\!_\fint$ of finitely atomic measures has the extreme point set $\ext P\!_\fint=P_1$ 
and $P\!_\fint$ is a face of the probability simplex $P$. In particular, $P\!_\fint$ is a 
linearly compact, $k$-neighborly simplex for every $k\in\N$. Indeed, having in mind 
$P\!_\fint=\conv P_1$, one gets $\ext P_\fint\subset P_1$. In addition, $P_1\subset\ext P$ 
holds by (2), hence $\ext P_\fint=P_1$ and it follows from Lemma~\ref{lem:basic-convexity}~(1) 
that $P\!_\fint$ is a face of $P$, since $P$ is a simplex by (1). The rest follows from
Lemma~\ref{lem:basic-convexity}. (More elementary proofs are possible.)
\end{enumerate}
\end{remark}
\begin{example}[Co-countable $\sigma$-algebra]\label{rem:co-countable}
By Remark~\ref{rem:measures-basics}~(2), every point measure is an extreme 
point of $P$. The \emph{co-countable $\sigma$-algebra} 
$\{A\subset\R:\text{$A$ or its complement is at most countable}\,\}$
admits extreme points of $P$ that are no point measures. An example is the 
measure that assigns the value $0$ to the at most countable sets and $1$ 
to the uncountable sets \cite[Section~6.9]{schilling-kuehn21}.
\end{example}
\begin{definition}[Moment constraints]\label{def:moment-constraints}
Let $\bp\colon X\to\R^d$ be measurable with respect to $\cF$ and let
\begin{align*}
P\!_\ait  & \textstyle 
:=\{ \mu\in P\!_\fint : 
\text{$\mu$ is concentrated on a finite subset $\{x_1,\dots,x_k\}\subset X$}\\
& \textstyle \hspace{\widthof{$:=\{$}}
\text{such that $\bp(x_1),\dots,\bp(x_k)$ are affinely independent}\,\}.
\end{align*}
We refer to the measures in $P\!_\ait$ as finitely atomic measures having the 
affine independence property. Let $C\subset\R^d$ be a  convex subset and let
\begin{align*}
G_\fint(C) & \textstyle
:= \{\mu\in P\!_\fint : \int\bp\dd\mu\in C\},\\
G_\ait(C) & \textstyle
:= \{\mu\in P\!_\ait : \int\bp\dd\mu\in C\}.
\end{align*}
Let $Q\subset P$ be a convex set of probability measures. Denoting the components of 
$\bp=(p_1,\ldots,p_d)$ by $p_i\colon X\to\R$, $i=1,\ldots,d$, we define
\begin{align*}
Q^\bp & \textstyle 
:= \{\mu\in Q : \text{$p_i$ is $\mu$-integrable $\forall i=1,\dots,d$}\,\},\\
G(C) & \textstyle
:= \{\mu\in Q^\bp : \int\bp\dd\mu\in C\}.
\end{align*}
We abbreviate $G_\fint(\bq):=G_\fint(\{\bq\})$, etc., for singletons 
$\{\bq\}$ with elements $\bq\in\R^d$.
\end{definition}
The measures in $P\!_\ait$ are concentrated on finite sets of sizes $d+1$ or less because 
of the affine independence property, which is a necessary condition for measures in $G(C)$ 
to be extreme points:
\par
\begin{theorem}[Necessary condition]\label{thm:moment-char-ext}
Consider the following assertions.
\begin{enumerate}
\item 
The convex set $Q$ is a face of $P$. 
Furthermore, we have $\ext Q\subset P_1$ or $\ext P=P_1$.
\item 
The convex set $Q$ is a simplex and $\ext Q\subset P_1$.
\item 
The convex set $Q$ is linearly compact and $\ext Q\subset P_1$.
\item 
We have $\ext G(C)\subset G_\ait(C)$.
\end{enumerate}
Then: (1)$\Rightarrow$(2)$\Rightarrow$(3)$\Rightarrow$(4).
\end{theorem}
\begin{proof}
(1)$\Rightarrow$(2). 
Since $Q$ is a face of $P$, we have $\ext Q\subset\ext P$. So $\ext P=P_1$
implies $\ext Q\subset P_1$. Since $P$ is a simplex, its face $Q$ is also 
a simplex, see Remark~\ref{rem:measures-basics}~(1) and 
Lemma~\ref{lem:basic-convexity}~(3a).
\par
(2)$\Rightarrow$(3).
Every simplex is linearly compact by Lemma~\ref{lem:basic-convexity}~(2). 
\par
(3)$\Rightarrow$(4).
It is easy to see that $P^\bp$ is a face of $P$. Hence, $Q^\bp$ is a face 
of $Q$. This has two consequences. First, $Q^\bp$ is linearly compact and 
secondly, $\ext[Q^\bp]\subset\ext Q\subset P_1$. The claim follows then from 
Proposition~\ref{pro:ext-fin}~(1) when $K:=Q^\bp$ and 
$\alpha\colon Q^\bp\to\R^d$, $\mu\mapsto\int\bp\dd\mu$. 
\end{proof}
The direct proof of (1)$\Rightarrow$(3) is much simpler than the proofs
(1)$\Rightarrow$(2)$\Rightarrow$(3) in Theorem~\ref{thm:moment-char-ext}. 
The affine independence property is also a sufficient condition for measures 
in $G(C)$ to be extreme points:
\par
\begin{theorem}[Sufficient condition]\label{thm:moment-char-fin}~
\begin{enumerate}
\item 
If $C$ is a singleton, then $G_\ait(C)\cap Q\subset\ext G(C)$.
\item
If $C$ is a singleton and $P_1\subset Q$, then $G_\ait(C)\subset\ext G(C)$.
\end{enumerate}
\end{theorem}
\begin{proof}
(1) First, $G_\ait(C)\subset\ext G_\fint(C)$ follows from 
Proposition~\ref{pro:ext-fin}~(2) applied to $K:=P\!_\fint$ and 
$\alpha\colon P\!_\fint\to\R^d$, $\mu\mapsto\int\bp\dd\mu$, because 
$\ext P\!_\fint=P_1$ and because $P\!_\fint$ is $(d+1)$-neighborly by
Remark~\ref{rem:measures-basics}~(3). Secondly, as $P\!_\fint$ is a 
face of $P$, the set 
\[
P\!_\fint\cap Q = P\!_\fint\cap Q^\bp
\]
is a face of $Q^\bp$. Hence $G_\fint(C)\cap Q$ is a face of $G(C)$
as it follows from imposing the affine constraints. Combining the two, 
we have
\[
G_\ait(C)\cap Q
\subset 
\ext[G_\fint(C)]\cap Q
\subset
\ext[G_\fint(C)\cap Q]
\subset
\ext G(C).
\]
Part~(2) is a corollary of part~(1) as 
$P_1\subset Q$ implies $G_\ait(C)\subset Q$.
\end{proof}
\begin{corollary}[Finitely atomic measures]\label{cor:fin-atomic-ext}~
\begin{enumerate}
\item 
The constrained set $G_\fint(C)$ is a face of $G(C)$ if $P_1\subset Q$.
\item
We have $\ext G_\fint(C)\subset G_\ait(C)$.
\item
We have $\ext G_\fint(C)=G_\ait(C)$ if $C$ is a singleton.
\end{enumerate}
\end{corollary}
\begin{proof}
Part~(1) is revealed in the proof of Theorem~\ref{thm:moment-char-fin}. 
Part~(2) and~(3) follow from Theorem~\ref{thm:moment-char-ext}~(3) and 
Theorem~\ref{thm:moment-char-fin}~(2) when $Q=P_\fint$ by 
Remark~\ref{rem:measures-basics}~(3).
\end{proof}
Theorem~\ref{thm:moment-char-fin} makes essentially one single assumption: that $C$ 
is a singleton. This is addressed in Example~\ref{exa:counter}~(2b) and~(2c) above. 
We now study the assumptions of Theorem~\ref{thm:moment-char-ext}.
\par
\begin{definition}\label{def:Borel}
The \emph{$\sigma$-algebra generated by} a family of subsets of $X$ is the smallest 
$\sigma$-algebra on $X$ that includes this family. A $\sigma$-algebra on $X$ is 
\emph{countably generated} if it is generated by a countable family of subsets of 
$X$. We assume topological spaces to be Hausdorff. The \emph{Borel $\sigma$-algebra} 
on a topological space $X$, denoted by $\cB(X)$, is the $\sigma$-algebra generated 
by the open sets. Its elements are \emph{Borel sets} and the measures on $\cB(X)$ 
are \emph{Borel measures}.
\end{definition}
\begin{remark}[On $\ext P=P_1$]\label{rem:two-valued}
A measure $\mu\in P$ is two-valued if and only if $\mu$ is an extreme point of $P$ by
Remark~\ref{rem:measures-basics}~(2). Hence, $\ext P=P_1$ holds if and only if every 
two-valued probability measure on $\cF$ is a point measure. This is true if
\begin{enumerate}
\item
$\cF$ is countably generated \cite[p.~14]{bhaskara-rao-rao81}, see also 
\cite[Section~6.10]{schilling-kuehn21}. For example, the Borel $\sigma$-algebra $\cB(X)$ 
is countably generated if $X$ is
\begin{enumerate}
\item
a separable metric space \cite[Theorem~11 in Chapter~4]{kelley75}.
\item
a Souslin space%
\footnote{A \emph{Polish space} is a topological space homeomorphic to 
a complete separable metric space. A \emph{Souslin space} is a Hausdorff 
space that is the image of a Polish space under a continuous map. Every 
Borel set in a Souslin space is itself a Souslin space 
\cite[p.~96]{schwartz73}. For example, every Borel set in the 
Polish space $\R^n$ is a Souslin space.}
\cite[p.~108]{schwartz73}. 
\end{enumerate}
\item
$\cF=\cB(X)$ for a metric space $X$ of the cardinality of the continuum 
\cite{marczewski-sikorski48}. An example is the nonseparable Banach space 
$X=\ell^\infty$ of bounded real sequences (see \cite{jech03}, Chapter~3).
\end{enumerate}
See \cite{winkler88}, Example~2.1, for further spaces of probability measures 
with extreme point set $P_1$.
\end{remark}
\begin{example}[Borel $\sigma$-algebras with $\ext P\neq P_1$]
The measure in Example~\ref{rem:co-countable} is a Borel measure for the topology on $\R$ 
whose closed sets are at most countable. This topology is not Hausdorff. But there are even 
compact Hausdorff spaces that admit two-valued Borel measures that are no point measures. 
An example is the \emph{Dieudonn\'{e} measure} \cite[Example~7.1.3]{bogachev07-vol2}. 
\end{example}
We present a criterion to identify faces of the probability simplex $P$.
\par
\begin{proposition}\label{pro:face-crit}
Let $\cN$ be a set of pairs $((A_i)_{i\in I},A)$ where $(A_i)_{i\in I}\subset\cF$ 
is an increasing net upper bounded by $A\in\cF$ (the directed set $I$ can 
depend on the pair). Then the set of probability measures $\mu\in P$ that satisfy 
$\mu(A)=\sup_{i\in I}\mu(A_i)$ for all $((A_i)_{i\in I},A)\in\cN$ is a face of $P$.
\end{proposition}
\begin{proof}
Let $\mu,\eta,\xi\in P$, $\lambda\in(0,1)$, and $\mu=(1-\lambda)\eta+\lambda\xi$. 
Let $((A_i)_{i\in I},A)\in\cN$. Since $(A_i)_{i\in I}$ is an increasing net, we
have
\[\textstyle
\sup_{i\in I}\mu(A_i)
=(1-\lambda)\sup_{i\in I}\eta(A_i)
+\lambda\sup_{i\in I}\xi(A_i).
\]
Since $A$ is an upper bound to $(A_i)_{i\in I}$, this implies that 
$\mu(A)=\sup_{i\in I}\mu(A_i)$ is equivalent to the conjunction of 
\[\textstyle
\eta(A)=\sup_{i\in I}\eta(A_i) 
\quad\text{and}\quad 
\xi(A)=\sup_{i\in I}\xi(A_i).
\qedhere
\]
\end{proof}
\begin{definition}[Special Borel measures]\label{def:spec-Borel}
Let $\mu$ be a Borel measure on a topological space $X$. Then $\mu$ is a 
\emph{Radon measure}%
\footnote{Radon measures are also known as \emph{tight} measures \cite{topsoe70}.}
\cite{schwartz73} if $\mu$ is inner regular on compact sets, that is to say, the 
measure $\mu(A)$ of any $A\in\cB(X)$ is the supremum of $\mu(B)$ over all compact 
subsets $B\subset A$. The measure $\mu$ is \emph{$\tau$-smooth}%
\footnote{It is well known that every Radon measure is $\tau$-smooth
\cite[Proposition~7.2.9]{bogachev07-vol2}. Also, for every $\tau$-smooth 
measure $\mu$ there exists a smallest closed set $A$ of full measure 
$\mu(A)=\mu(X)$, called the \emph{support} of $\mu$.}
if $\mu(\bigcup_{i\in I}A_i)=\sup_{i\in I}\mu(A_i)$ for every increasing net 
$(A_i)_{i\in I}$ of open sets \cite{topsoe70,winkler85}. Let $R$ and $P\!_\tau$
be the sets of Radon and $\tau$-smooth probability measures on $X$, respectively.
\end{definition}
\begin{corollary}[Radon and $\tau$-smooth measures]\label{cor:Borel-faces}
Let $X$ be a topological space.
\begin{enumerate}
\item 
The sets $R$ and $P\!_\tau$ are faces of the simplex $P$ of Borel probability measures.
They are linearly compact, $k$-neighborly simplices for every $k\in\N$. 
\item 
The extreme points are $\ext R=\ext P\!_\tau=P_1$ (Tops{\o}e).
\end{enumerate}
Consider moment constraints on the ambient space $Q=R$ or $Q=P\!_\tau$:
\begin{enumerate}
\item[(3)]
We have $\ext G(C)\subset G_\ait(C)$.
\item[(4)] 
We have $\ext G(C)=G_\ait(C)$ if $C$ is a singleton.
\item[(5)]
The constrained set $G_\fint(C)$ of finitely atomic measures is a face of $G(C)$.
\end{enumerate}
\end{corollary}
\begin{proof}
Part~(1). That $R$ is a face of $P$ follows from Proposition~\ref{pro:face-crit} because
the compact subsets of a Borel set are directed by inclusion, as the union of two compact 
sets is compact. Regarding the $\tau$-smooth measures, we note that every net of open 
sets is upper bounded by its union, which is open, hence Borel. The remaining statements 
follow from Remark~\ref{rem:measures-basics}~(1) and Lemma~\ref{lem:basic-convexity}.
\par
Part~(2). Tops{\o}e \cite[Theorem~11.1]{topsoe70} has shown $\ext P\!_\tau=P_1$.
The set $R$ of Radon measures is a face of $P$ by (1). Hence $R$ is a face of $P_\tau$.
This shows $\ext R\subset P_1$, hence $\ext R=P_1$.
\par
Part~(3) follows from Theorem~\ref{thm:moment-char-ext}~(1) as $Q$ is a face of 
$P$ and $\ext Q=P_1$, by part~(1) and~(2).
\par
Part~(4) is implied by part~(3) and Theorem~\ref{thm:moment-char-fin}~(2),
because $P_1\subset Q$ by part~(2).
\par
Part~(5) follows from Corollary~\ref{cor:fin-atomic-ext}~(1) as $P_1\subset Q$
by part~(2). 
\end{proof}
We recall that every Borel measure on any Borel set in $\R^n$, or more generally on any 
Souslin space, is a Radon measure \cite[p.~122]{schwartz73}, hence $R=P\!_\tau=P$ 
holds for these spaces.
\par
\begin{remark}[Summary]\label{rem:ext-moment-contribution}
The extreme points of the convex set $G(C)$ defined by $d<\infty$ many moment constraints 
on a convex set $Q\subset P$ have been characterized for the simplex $Q=R$ of Radon measures 
in \cite[Proposition~5]{weizsaecker-winkler80} and more generally for simplices $Q$ with 
the extreme point set $P_1$ in \cite[Theorem~2.1]{winkler88}. 
(The set $R$ of Radon probability measures is indeed a simplex with extreme point set $P_1$ 
by Corollary~\ref{cor:Borel-faces}~(1) and~(2) above.)
\begin{enumerate}
\item
We advance the sufficient condition for points of $G(C)$ to be extreme points by proving 
that the assumption of $Q$ being a simplex is a redundant. This is achieved in 
Theorem~\ref{thm:moment-char-fin}, whose main idea is that the set $P\!_\fint$ of finitely 
atomic measures is a face of $P$.
\item  
Regarding the converse, necessary condition, we introduce the novel assumption that $Q$ is 
a face of $P$ in Theorem~\ref{thm:moment-char-ext}. This assumption is arguably quite strong, 
even stronger than $Q$ being a simplex. Its advantage is that no further assumptions are 
needed provided $\ext P=P_1$ (see Remark~\ref{rem:two-valued} for examples). Important 
examples of faces are the sets $R$ and $P\!_\tau$ of Radon and $\tau$-smooth measures, which 
are faces of the simplex $P$ of Borel probability measures by Proposition~\ref{pro:face-crit} 
and Corollary~\ref{cor:Borel-faces}.
\end{enumerate}
\end{remark}
\begin{remark}[Compact metric spaces]\label{rem:Karr}
Karr \cite{karr83} proves $\ext G(\bq)=G_\ait(\bq)$, $\bq\in\R^d$, for the set $Q=P$ of Borel 
probability measures on a compact metric space $X$ assuming $\bp\colon X\to\R^d$ is continuous. 
His proof uses the idea  \cite{douglas64} that $\mu\in G(\bq)$ is an extreme point if and only 
if $p_1,\dots,p_d$ and the constant map $X\to\R$, $x\mapsto 1$ span the Lebesgue space $L_1(\mu)$. 
Karr's result follows also from Theorem~\ref{thm:moment-char-ext}~(1) and 
Theorem~\ref{thm:moment-char-fin}~(2) as every compact metric space is a Polish space 
\cite[Theorem~32 in Chapter~6]{kelley75}, so $\ext P=P_1$ holds by Remark~\ref{rem:two-valued}~(1a).
\end{remark}
%
%
\subsection{Linear programs on measures}
\label{sec:measure-lp}
We study the optimal value of an integral functional on the constraint set $G(C)$ 
defined by finitely many moment constraints on a convex set $Q$ of probability 
measures. This optimal value is known to coincide with the optimal value on the 
extreme points of $G(C)$ and with the optimal value on finitely atomic measures 
having the affine independence property, if $C$ is a closed convex set and $Q=R$ 
is the simplex of Radon probability measures \cite{weizsaecker-winkler80,winkler88}. 
We generalize this result by replacing the assumption $Q=R$ with the much weaker
assumption that $Q$ includes the set $P_1$ of point measures.
\par
\begin{definition}[Generalized moment problems]\label{def:moment-problems}
Building on Definition~\ref{def:moment-constraints}, we assume that $\bp\colon X\to\R^d$ 
and $f\colon X\to(-\infty,+\infty]$ are measurable with respect to the $\sigma$-algebra 
$\cF$. We define the linear optimization problems
\begin{align}
\label{eq:PM-atomic}
g_\fint(C) & \textstyle
:= \inf\{ \int f\dd\mu : \mu\in G_\fint(C) \},\\
\label{eq:PM-aff-ind}
g_\ait(C) & \textstyle
:= \inf\{ \int f\dd\mu : \mu\in G_\ait(C) \}
\end{align}
on the spaces $G_\ait(C)\subset G_\fint(C)$ of finitely atomic measures. If $f$ is 
bounded below by a real constant, then we also define
\begin{align}
\label{eq:PM}
g(C) & \textstyle
:= \inf\{ \int f\dd\mu : \mu\in G(C) \},\\
\label{eq:PM-ext}
g_\extt(C) & \textstyle
:= \inf\{ \int f\dd\mu : \mu\in \ext G(C) \}.
\end{align}
The optimization problem $g(\bq)$ in \eqref{eq:PM} at a specific point $\bq\in\R^d$ 
is an example of the problem of moments \cite{shapiro01}. This problem is also known 
as the generalized moment problem if $\cF=\cB(X)$ is the Borel $\sigma$-algebra on a Borel 
\cite{lasserre10}, closed \cite{henrion-etal20}, or compact \cite{klerk-laurent19} set 
$X$ in $\R^k$. Loosely speaking, we refer to each of 
\eqref{eq:PM-atomic}--\eqref{eq:PM-ext} as a \emph{generalized moment problem}.
\end{definition}
Clearly, $g(C)\leq g_\extt(C)$. An obvious reason for a strict inequality 
$g(C)<g_\extt(C)$ is a lack of extreme points. Also, $g_\ait(C)\leq g(C)$ 
holds by the Richter-Tchakaloff theorem without any assumptions. 
A strict inequality $g_\ait(C)<g(C)$ can only occur if $P_1\subset Q$ fails.
\par
\begin{example}[No extreme points]
The set $Q$ of Bernoulli distributions with uncertain outcomes is an open segment. 
In absence of constraints 
$G(C)=Q=\{(1-\lambda)\delta_0+\lambda\delta_1:0<\lambda<1\}$ holds. 
So $g_\extt(C)=+\infty$ whereas $g(C)=\min\{f(0),f(1)\}$ is finite if 
$f\colon\{0,1\}\to\R$ is finite. Note that $g_\ait(C)=g(C)$ despite $P_1\not\subset Q$.
\end{example}
\begin{example}[No finitely atomic measures]
The set $Q$ of Borel probability measures absolutely continuous with respect to  
Lebesgue measure on the unit interval $[0,1]$ is a face of the set $P$ of Borel 
probability measures on $[0,1]$, see for example \cite[Lemma 9.1]{weis25}. The 
convex set $Q$ has no extreme points because $\ext P=P_1$ by 
Remark~\ref{rem:two-valued}~(1a). Hence, in absence of constraints,
\[
g_\ait(C)=\inf f<g(C)=0<g_\extt(C)=+\infty,
\]
if $f\colon[0,1]\to\R$ is zero except on a set of Lebesgue measure zero and 
$\inf f<0$.
\end{example}
\begin{figure}
\hfill
\begin{tikzpicture}[every node/.style={draw, rectangle, rounded corners=2mm, inner sep=7pt,
minimum width=5cm, align=center}]
\draw[white](-6.0,0) rectangle (6.0,7.0);
\tkzDefPoint(-6.0,7.0){c11}
\tkzDefPoint(6.0,7.0){c12}
\tkzDefPoint(-6.0,2.4){c21}
\tkzDefPoint(6.0,2.4){c22}
\tkzDefPoint(0,0){c3}
\draw (c11) node[anchor=north west] (n11)
{$C$ is closed};
\draw (c12) node[anchor=north east] (n12)
{$Q$ is the simplex $R$\\of Radon measures};
\draw (c21) node[anchor=south west, rectangle split, rectangle split parts=2] (n21)
{$g_\extt(C)\leq g(C)$\nodepart{two}Weizs\"{a}cker and Winkler's\\integral representation};
\draw (c22) node[anchor=south east, rectangle split, rectangle split parts=2] (n22)
{$g_\ait(C)\leq g_\extt(C)$\nodepart{two}
necessary condition};
\draw (c3) node[anchor=south] (n3) 
{$g(C)=g_\fint(C)=g_\ait(C)=g_\extt(C)$};
\draw[thick,shorten >=2mm,shorten <=2mm,-latex] (n11) -- (n21);
\draw[thick,shorten >=2mm,shorten <=2mm,-latex] (n12.south west) -- (n21.north east);
\draw[thick,shorten >=2mm,shorten <=2mm,-latex] (n12) -- (n22);
\draw[thick,shorten >=2mm,shorten <=2mm,-latex] (n21) -- (n3);
\draw[thick,shorten >=2mm,shorten <=2mm,-latex] (n22) -- (n3);
\end{tikzpicture}%
\hspace*{\fill}
\caption{\label{fig:flowchart1}%
Flowchart for the proof of Theorem~\ref{thm:WW-integral}.
The necessary condition refers to the extreme points of $G(C)$
being finitely atomic measures that have the affine independence
property.}
\end{figure}
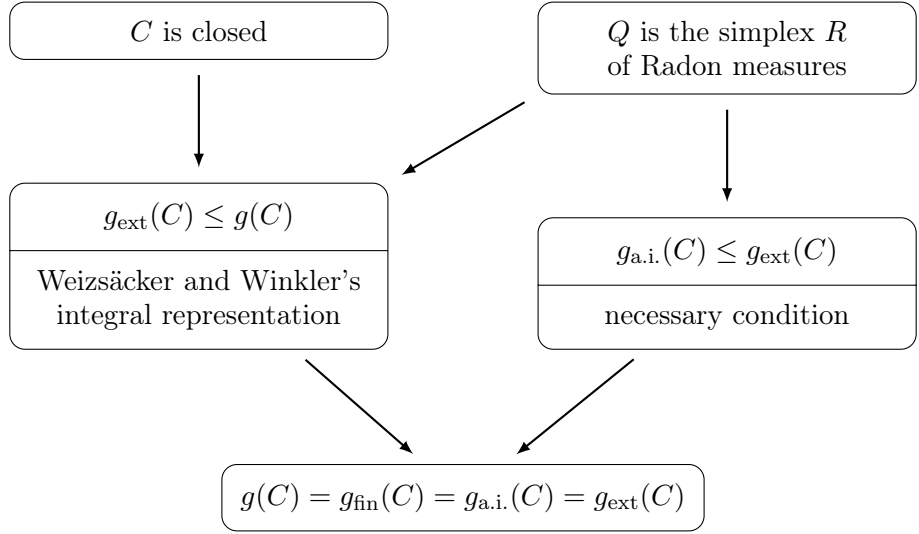
\begin{theorem}[Weizs\"{a}cker-Winkler]\label{thm:WW-integral}
Let $C$ be closed and $Q=R$ be the simplex of Radon probability measures on a Hausdorff 
space $X$. If $\bp\colon X\to\R^d$ and $f\colon X\to(-\infty,+\infty]$ are Borel measurable, 
and $f$ is bounded below by a real constant, then $g(C)=g_\fint(C)=g_\ait(C)=g_\extt(C)$.
\end{theorem}
\begin{proof}
The claim $g_\extt(C)\leq g(C)$ follows by monotone convergence from the integral 
representation for Radon probability measures \cite{weizsaecker-winkler79}. The proof is then 
completed by 
\[
g_\ait(C)\leq g_\extt(C)\leq g(C)\leq g_\fint(C)\leq g_\ait(C),
\]
as $\ext G(C)\subset G_\ait(C)$ holds by Corollary~\ref{cor:Borel-faces}~(3).
See Figure~\ref{fig:flowchart1}.
\par
In more detail, let $f_1,f_2,\dots$ be a countable family of Borel measurable functions 
$X\to\R$. Let $I$ be an index set, $\tilde C\subset\R^I$ a closed convex set, and for 
every $i\in I$ let $g_i\colon X\to\R$ be continuous for the topology generated%
\footnote{%
The topology on $X$ \emph{generated by} a family of subsets of $X$ is the coarsest 
topology for which all members of this family are open. The \emph{initial topology} of 
$f_1,f_2,\dots$ is the topology on $X$ generated by the preimages $f_j^{-1}(U)$, where 
$U$ is any open set in $\R$ and $j\in\N$.}
by the given topology of $X$ and the initial topology of $f_1,f_2,\dots$. To guarantee a 
certain degree of boundedness of the (possibly uncountably many) constraints $g_i$, we 
assume that for every $i\in I$ there is $j\in\N$ and a real constant $K>0$ such that 
$|g_i|\leq K(1+|f_j|)$ pointwise. Here, $|h|\colon X\to\R$ is defined by $|h|(x)=|h(x)|$, 
$x\in X$, for every $h\colon X\to\R$. Let
\[\textstyle
\tilde H:=\{\nu\in R \mid
\text{$\forall j\in\N:f_j\subset L_1(\nu)$ 
and 
$(\int g_i\dd\nu)_{i\in I}\in\tilde C$} \}
\]
and let $\Sigma(\ext\tilde H)$ be the smallest $\sigma$-algebra on $\ext\tilde H$ such that 
$\ext\tilde H\to\R$, $\mu\mapsto\mu(B)$ is measurable for every Borel set $B$ in $X$. Then 
Corollary~2 in \cite{weizsaecker-winkler79} shows that for every measure $\mu\in\tilde H$ 
there exists a probability measure $\Pi$ on $\Sigma(\ext\tilde H)$, such that
\begin{equation}\label{eq:measure}\textstyle
\mu(B)=\int_{\ext\tilde H} \nu(B)\dd\Pi(\nu),
\quad B\in\cB(X).
\end{equation}
The generalized moment problem \eqref{eq:PM} with $Q=R$ fits into this setup when 
$I:=\{1,\dots,d\}$, $f_i:=g_i:=p_i$ for every $i\in I$, and $\tilde C:=C$. Equation 
\eqref{eq:measure} shows that for every $\mu\in G(C)$ there is a probability measure 
$\Pi$ on $\Sigma(\ext G(C))$ such that 
\begin{equation}\label{eq:integral}\textstyle
\int_X f\dd\mu=\int_{\ext G(C)}(\int_X f\dd\nu)\dd\Pi(\nu)
\end{equation}
where $f=1_B$ is the indicator function of a Borel set $B\in\cB(X)$, defined by 
$1_B(x):=\delta_x(B)$, $x\in X$. Equation \eqref{eq:integral} holds for every Borel 
measurable $f\colon X\to(-\infty,+\infty]$ that is bounded below by a real constant, 
by monotone convergence, see for example \cite{bauer01}, Sections~11--12. Hence 
$g_\extt(C)\leq g(C)$ as $\mu\in G(C)$ is arbitrary and because the monotonicity of 
the integral shows
\[\textstyle
g_\extt(C)
=\inf_{\nu\in\ext G(C)}\int_X f\dd\nu
\leq \int_{\ext G(C)}(\int_X f\dd\nu)\dd\Pi(\nu)
= \int_X f\dd\mu.
\qedhere
\]
\end{proof}
\begin{remark}[Compactness]\label{rem:no-compactness}
The integral representation \eqref{eq:measure} is beyond standard Choquet theory \cite{phelps01} 
as it assumes no compactness of $\tilde H$. Of course, the simplex $R$ of Radon probability 
measures is compact in special cases. For example, if $X$ is a compact Hausdorff space, then 
$R$ is weakly-* compact \cite[Corollary~13 in Chapter~21]{royden-fitzpatrick10}. If, in addition, 
the constraints $p_1,\dots,p_d$ are bounded continuous functions and $C\subset\R^d$ is closed, 
then $\tilde H$ is weakly-* compact.
\end{remark}
\begin{figure}
\hfill
\begin{tikzpicture}[every node/.style={draw, rectangle, rounded corners=2mm, inner sep=7pt,
minimum width=5cm, align=center}]
\draw[white](-8.0,0) rectangle (8.0,8.0);
\tkzDefPoint(-3.75,8.0){d01}
\tkzDefPoint(1.0,8.0){d02}
\tkzDefPoint(-3.75,6.2){d11}
\tkzDefPoint(8.0,6.2){d12}
\tkzDefPoint(-8.0,2.0){d21}
\tkzDefPoint(-1.0,2.0){d22}
\tkzDefPoint(6.0,2.0){d23}
\tkzDefPoint(0,0){d3}
\draw (d01) node[anchor=north] (m01)
{$C$ is closed};
\draw (d02) node[anchor=north, minimum width=2.5cm] (m02)
{$Q$ includes\\the set $P_1$\\of point\\measures};
\draw (d11) node[anchor=north, rectangle split, rectangle split parts=2] (m11)
{$g_{\ext,\fint}(C)\leq g_\fint(C)$\nodepart{two}Bauer's theorem};
\draw (d12) node[anchor=north east, rectangle split, rectangle split parts=2] (m12)
{$g_\fint(C)\leq g(C)$\nodepart{two}Richter-Tchakaloff\\theorem};
\draw (d21) node[anchor=south west, rectangle split, rectangle split parts=2, minimum width=4cm] (m21)
{$g_\ait(C)\leq g_\fint(C)$\nodepart{two}necessary condition};
\draw (d22) node[anchor=south, rectangle split, rectangle split parts=2, minimum width=4cm] (m22)
{$g_\extt(C)\leq g_\fint(C)$\nodepart{two}facial geometry};
\draw (d23) node[anchor=south east, minimum width=4cm] (m23)
{$g_\fint(C)=g(C)$};
\draw (d3) node[anchor=south] (m3) 
{$g(C)=g_\fint(C)=g_\ait(C)=g_\extt(C)$};
\draw[thick,shorten >=2mm,shorten <=2mm,-latex] (m01) -- (m11);
\draw[thick,shorten >=2mm,shorten <=2mm,-latex] (m02) -- ([xshift=-1.0cm]m23.north);
\draw[thick,shorten >=2mm,shorten <=2mm,-latex] (m02) -- (m22);
\draw[thick,shorten >=2mm,shorten <=2mm,-latex] (m11) -- (m22);
\draw[thick,shorten >=2mm,shorten <=2mm,-latex] (m11) -- (m21);
\draw[thick,shorten >=2mm,shorten <=2mm,-latex] (m12) -- ([xshift=-1.0cm]m23.north);
\draw[thick,shorten >=2mm,shorten <=2mm,-latex] (m23.south west) -- (m3);
\draw[thick,shorten >=2mm,shorten <=2mm,-latex] (m22) -- (m3);
\draw[thick,shorten >=2mm,shorten <=2mm,-latex] (m21.south east) -- (m3);
\end{tikzpicture}%
\hspace*{\fill}
\caption{\label{fig:flowchart2}%
Flowchart for the proof of Theorem~\ref{thm:RTR}. Bauer's theorem is applied to 
constrained $k$-simplices, $k\in\N_0$, and $g_{\ext,\fint}(C)$ is the optimal 
value on the extreme points of $G_\fint(C)$. The necessary condition refers to 
the extreme points of $G_\fint(C)$ having the affine independence property. }
\end{figure}
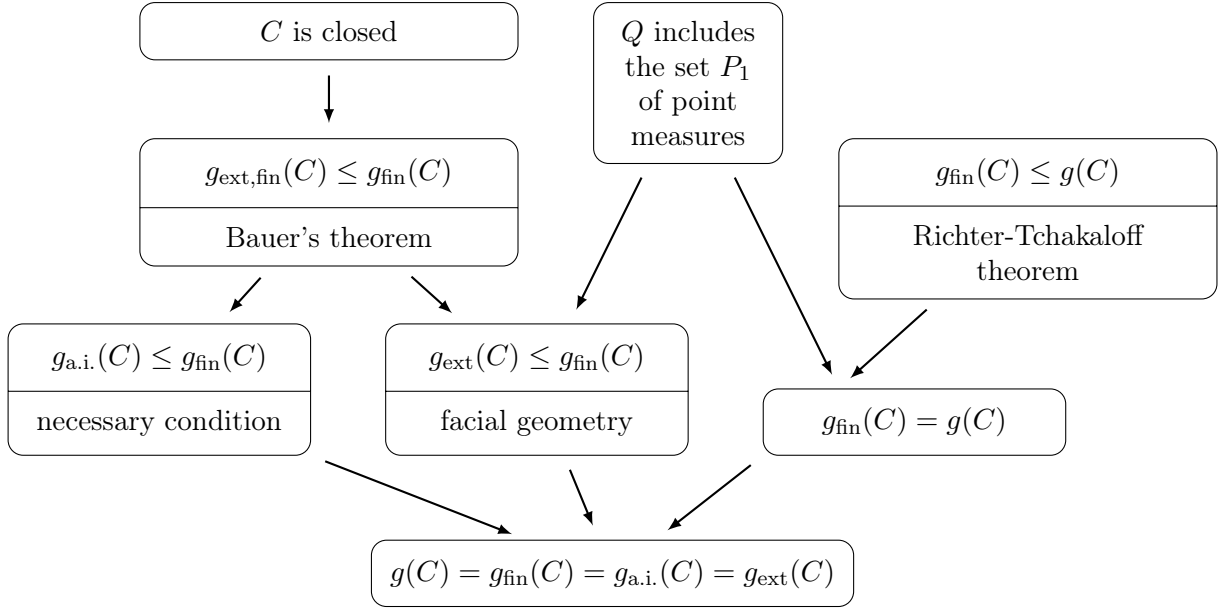
We generalize Theorem~\ref{thm:WW-integral}. Referring to $\bp\colon X\to\R^d$ and 
$f\colon X\to(-\infty,+\infty]$ as arbitrary maps, we tacitly understand that $\cF$ 
is the discrete $\sigma$-algebra (the set of all subsets of $X$).
\par
\begin{theorem}\label{thm:RTR}
Let $C$ be closed and $P_1\subset Q$. If $\bp\colon X\to\R^d$ and 
$f\colon X\to(-\infty,+\infty]$ are arbitrary maps then $g_\fint(C)=g_\ait(C)$. If 
$\bp$ and $f$ are measurable with respect to $\cF$, and $f$ is bounded below by a 
real constant, then $g(C)=g_\fint(C)=g_\ait(C)=g_\extt(C)$.
\end{theorem}
\begin{proof}
See Figure~\ref{fig:flowchart2} for a flowchart of the proof.
First, we focus on finitely atomic measures. Let $\mu\in G_\fint(C)$ be a finitely 
atomic measure concentrated on $\{x_1,\dots,x_k\}\subset X$. Since $C$ is closed 
and the convex hull $F'$ of the point measures at $x_1,\dots,x_k$ is compact, it 
follows that the constrained set $F:=F'\cap G_\fint(C)$ is compact. Bauer's maximum 
principle then shows that $\tilde\mu\mapsto\int f\dd\tilde\mu$ attains a minimum on 
$F$ at an extreme point $\nu$ of $F$. Note that $\nu$ is an extreme point of 
$G_\fint(C)$, because $F'$ is a face of the simplex $P$ and so $F$ is a face of 
$G_\fint(C)$. Hence
\begin{equation}\label{eq:bauer}\textstyle
g_{\ext,\fint}(C)
:=\inf\{ \int f\dd\mu : \mu\in\ext G_\fint(C)\}
\leq g_\fint(C).
\end{equation}
The inclusion $\ext G_\fint(C)\subset G_\ait(C)$, proved in 
Corollary~\ref{cor:fin-atomic-ext}~(2), implies $g_\ait(C)\leq g_{\ext,\fint}(C)$. 
Hence $g_\ait(C)\leq g_\fint(C)$ by \eqref{eq:bauer}, hence $g_\ait(C)=g_\fint(C)$.
\par
Secondly, to study arbitrary measures in $P$, we assume that $\bp$ and $f$ are 
measurable and $f$ is bounded below. Let $\bq\in\R^d$ and $\mu\in G(\bq)$. Then 
$\mu\in Q$ and $\bq=\int\bp\dd\mu$. If $c:=\int f\dd\mu$ is finite, then the 
Richter-Tchakaloff theorem \cite[Theorem 1.24]{schmuedgen17}, \cite[Lemma~3.1]{shapiro01} 
shows that there is a finitely atomic measure $\nu\in P\!_\fint$ such that 
$\bq=\int\bp\dd\nu$ and $c=\int f\dd\nu$. Hence, 
\begin{align*}
g_\fint(\bq)
&=\textstyle  
 \inf\{\tilde c\in\R\mid \exists \nu\in P\!_\fint
 : \bq=\int\bp\dd\nu, \tilde c=\int f\dd\nu\}\\
&\leq\textstyle  
 \inf\{\tilde c\in\R \mid \exists \mu\in Q
 : \bq=\int\bp\dd\mu, \tilde c=\int f\dd\mu\}
 =g(\bq).
\end{align*} 
The converse inequality $g_\fint(\bq)\geq g(\bq)$ holds, as $P_1\subset Q$ 
implies $P\!_\fint\subset Q$, so $g(\bq)=g_\fint(\bq)$. Taking the infimum 
over all points $\bq\in C$ shows $g(C)=g_\fint(C)$.
\par
Now, the set $G_\fint(C)$ of finitely atomic measures is a face of $G(C)$ by 
Corollary~\ref{cor:fin-atomic-ext}~(1) because $P_1\subset Q$. This shows 
$\ext G_\fint(C)\subset\ext G(C)$ and $g_\extt(C)\leq g_{\ext,\fint}(C)$. 
Hence $g_\extt(C)\leq g_\fint(C)$ by \eqref{eq:bauer}. Finally, 
\[
g_\extt(C)
\leq g_\fint(C)
= g(C)
\leq g_\extt(C).
\qedhere
\]
\end{proof}
The extreme points of the constraint set $G(C)$ are finitely atomic measures that have the 
affine independence property if $Q=R$, by Corollary~\ref{cor:Borel-faces}~(3). Whereas 
$\ext G(C)\subset G_\ait(C)$ can be false under the assumptions of Theorem~\ref{thm:RTR}, 
we still have $g_\ait(C)=g_\extt(C)$.
\par
\begin{remark}[Summary]\label{rem:our-contrib-opti}
Theorem~\ref{thm:WW-integral} has been proven in the corollary in 
\cite[p.~138]{weizsaecker-winkler80} and in \cite[Theorem~3.2]{winkler88}. We generalize
it in Theorem~\ref{thm:RTR} by replacing $Q=R$ with the weaker assumption that $Q$ 
includes the point measures $P_1$. The main ideas are summarized in 
Figure~\ref{fig:flowchart2}:
\begin{enumerate}
\item
$g_\fint(C)\leq g(C)$ by the Richter-Tchakaloff theorem,
\item
$g_\ait(C)\leq g_\fint(C)$ by Bauer's theorem and the necessary
condition,
\item
$g_\extt(C)\leq g_\fint(C)$ by Bauer's theorem and facial geometry.
\end{enumerate}
It is tempting to prove Theorem~\ref{thm:RTR} using (1) and (2) while replacing (3) 
with $g_\extt(C)\leq g_\ait(C)$, which is implied by the sufficient condition 
$G_\ait(C)\subset\ext G(C)$ in Theorem~\ref{thm:moment-char-fin}. However, the 
sufficient condition assumes that $C$ is a singleton. This underlines the advantage 
of the facial geometry in (3), which allows $C\subset\R^d$ to be an arbitrary closed 
convex set.
\end{remark}
%
%
\subsection{Moment problems and convex envelopes}
\label{sec:Dacorogna-MP}
We consider parametrized families of generalized moment problems whose constraints are 
defined by points (Definition~\ref{def:moment-problems}). We use the symbols of the 
optimal values also for the optimal value functions, e.g.\
$g_\fint\colon\R^d\to[-\infty,+\infty]$, $\bq\mapsto g_\fint(\bq)=g_\fint(\{\bq\})$. 
We prove that the optimal value functions are generalized convex envelopes. An example 
is the polyconvex envelope in elasticity theory \cite{ball76,dacorogna08}.
\par
\begin{definition}
Let $\bp\colon X\to\R^d$ and $f\colon X\to(-\infty,+\infty]$ be arbitrary maps.
\begin{enumerate}
\item
We say that $f$ is \emph{$\bp$-convex} if $f=h\circ\bp$ is the composition of $\bp$ 
followed by  a convex function $h\colon\R^d\to(-\infty,+\infty]$.
\item
We define the \emph{$\bp$-convex envelope} $\env_\bp f\colon X\to[-\infty,+\infty]$
of $f$ pointwise by 
\[
\env_\bp f(x) := \sup\{h(x): h\leq f, \text{$h$ is $\bp$-convex\,}\},
\qquad x\in X.
\]
As before, inequalities of functions are understood pointwise.
\end{enumerate}
\end{definition}
\begin{example}[Polyconvexity]\label{exa:polyconvexity}
Let $X:=\R^{m\times n}$, $d:=\binom{m+n}{m}-1$, and let $\bp\colon\R^{m\times n}\to\R^d$ be 
the map that returns all minors of an $m\times n$-matrix. Then the $\bp$-convex functions 
are the \emph{polyconvex} functions introduced by J.\,M.~Ball \cite{ball76}. The $\bp$-convex 
envelope of $f\colon\R^{m\times n}\to(-\infty,+\infty]$ is the \emph{polyconvex envelope} 
\cite{dacorogna08}, which is frequently denoted by $Pf:=\env_\bp f$. 
\par 
Polyconvex energy densities play a central role in nonlinear elasticity because they allow 
one to impose physically meaningful assumptions while guaranteeing mathematical 
well-posedness. One aims at minimizing the stored energy $\int_{\Omega}f(\nabla y(x))\dd x$
over admissible deformations $y\colon\Omega\to\R^3$, given a reference configuration 
$\Omega\subset\R^3$, an energy density $f$, and boundary conditions. For example, an 
energy-minimizing deformation exists if $f$ is polyconvex and satisfies certain growth 
conditions at infinity, see for example Section~8.4.2 \cite{dacorogna08}. A physically 
meaningful constraint is to penalize small determinants of  $\nabla y$ with a high energy density 
(and exclude negative determinants of $\nabla y$) to preserve the orientation. When $f$ is not polyconvex, its 
polyconvex envelope $Pf$ provides the closest relaxation that still respects physical 
constraints. The convex envelope is too rigid for modeling microstructure, already because 
it cannot have multiple local minima. This makes the polyconvex envelope an effective tool 
for modeling microstructure and for ensuring mathematically robust solutions in nonlinear 
elasticity \cite{ciarlet21,kruzik-roubicek19}.
\end{example}
\begin{example}[Electro-magneto-elasticity]\label{exa:Silhavy}
Polyconvexity can be generalized to the so-called $\cA$-polyconvexity. Roughly speaking, it is an analog of classical polyconvexity, but the gradients are replaced by a more general first-order differential operator $\cA$, and we look at maps satisfying $\cA u=0$. In fact, gradients can be described by $\cA=\mathrm{curl}$, because $\mathrm{curl}\,\nabla $=0.  
The notion of $\cA$-polyconvexity was used in \cite{silhavy18} where the author discusses variational problems in electromagnetism.
\end{example}
\begin{proposition}\label{pro:Dacorogna}
Let $\bp\colon X\to\R^d$ and $f\colon X\to(-\infty,+\infty]$ be arbitrary maps. The optimal 
value function $g_\fint\colon\R^d\to[-\infty,+\infty]$ is convex and satisfies $g_\ait=g_\fint$. 
If $f$ is bounded below by a $\bp$-convex function, then $g_\fint>-\infty$ pointwise and 
$\env_\bp f=g_\fint\circ\bp$. 
\end{proposition}
\begin{proof}
Note that $\mu\mapsto\int f\dd\mu$ and $\mu\mapsto\int p_i\dd\mu$, $i=1,\dots,d$ 
are real affine maps on the convex set 
$\tilde P\!_\fint:=\{\mu\in P\!_\fint:\int f\dd\mu<\infty\}$. Hence, 
$\tilde g_\fint\colon \R^d\to[-\infty,+\infty]$,
\[\textstyle
\tilde g_\fint(\bq):=
\inf\{\int f\dd\mu : \bq=\int\bp\dd\mu, \mu\in\tilde P\!_\fint\},
\qquad \bq\in\R^d,
\]
is convex by Theorem~1 and Example~1 in \cite{rockafellar74}. And so is 
$g_\fint=\tilde g_\fint$. The equality $g_\ait=g_\fint$ is shown in
Theorem~\ref{thm:RTR} as the singletons $\{\bq\}$ are closed for all
$\bq\in\R^d$.
\par
Let $f$ be bounded below by a $\bp$-convex function, say $h\circ\bp\leq f$ and
$h\colon\R^d\to(-\infty,+\infty]$ is convex. Jensen's inequality 
\cite[Theorem~4.3]{rockafellar70} shows that for every $\mu\in P\!_\fint$ we have
\[\textstyle
h(\int\bp\dd\mu)
\leq \int h\circ\bp\dd\mu
\leq \int f\dd\mu.
\]
Minimizing for a fixed $\bq\in\R^d$ the right-hand side of $h(\bq)\leq\int f\dd\mu$
over all $\mu\in P\!_\fint$ that satisfy $\bq=\int\bp\dd\mu$, we get 
$-\infty<h\leq g_\fint$ pointwise. For every $x\in X$, the constraint 
$\bp(x)=\int\bp\dd\mu$ is satisfied by the point measure $\mu=\delta_x$, hence
\[\textstyle
h\circ\bp
\leq g_\fint\circ\bp
\leq f.
\]
Thus, $g_\fint\circ\bp$ is a $\bp$-convex lower bound to $f$, so 
$g_\fint\circ\bp\leq \env_\bp f$. Maximizing for a fixed $x\in X$ the left-hand 
side of $h\circ\bp(x)\leq g_\fint\circ\bp(x)$ over all convex functions 
$h\colon\R^d\to(-\infty,+\infty]$ that satisfy $h\circ\bp\leq f$, we conclude 
$\env_\bp f\leq g_\fint\circ\bp$.
\end{proof}
\begin{example}[Convex envelope]\label{exa:convex-env}
Let $X:=\R^d$ and let $\bp\colon\R^d\to\R^d$ be the identity map. Then the $\bp$-convex 
functions are simply the convex functions $\R^d\to(-\infty,+\infty]$. The $\bp$-convex 
envelope of $f\colon\R^d\to(-\infty,+\infty]$ is simply the \emph{convex envelope} 
\cite{dacorogna08}, which is frequently denoted by $Cf:=\env_\bp f$. We have 
$C f\leq\conv f$, where $\conv f\colon\R^d\to[-\infty,+\infty]$ is the greatest convex 
function below $f$ (the \emph{convex hull} of $f$). It is well known that
\begin{equation}\label{eq:convex-hull}\textstyle
\conv f(x)=\inf\{\sum_{i=1}^{d+1}\lambda_if(x_i) : x=\sum_{i=1}^{d+1}\lambda_ix_i\},
\qquad x\in\R^d,
\end{equation}
where the infimum is taken over all expressions of $x$ as a convex combination of $d+1$ points. 
The formula \eqref{eq:convex-hull} is also valid if one takes only the combinations in which 
the points $x_i$ whose weights $\lambda_i$ are strictly positive are affinely independent 
\cite[Corollary~17.1.5]{rockafellar70}. If $f$ is bounded below by a convex function 
$\R^d\to(-\infty,+\infty]$, then $C f=\conv f$ and the formula \eqref{eq:convex-hull} and its 
affine independence statement follow immediately from Proposition~\ref{pro:Dacorogna}.
\end{example}
\begin{corollary}\label{cor:Dacorogna}
Let $P_1\subset Q$. If $\bp$ and $f$ are measurable with respect to $\cF$, and $f$ is bounded 
below by a real constant, then $g=g_\fint=g_\ait=g_\extt$ and $\env_\bp f=g_\fint\circ\bp$.
\end{corollary}
\begin{proof}
This follows immediately from Proposition~\ref{pro:Dacorogna} and Theorem~\ref{thm:RTR}.
\end{proof}
\begin{example}[Old and new representations of the polyconvex envelope]%
\label{exa:polyconvex-envelope}
Let $\bp\colon\R^{m\times n}\to\R^d$ be the minor map of Example~\ref{exa:polyconvexity}. 
Dacorogna has shown that if $f\colon\R^{m\times n}\to(-\infty,+\infty]$ is bounded below 
by a polyconvex function, then the polyconvex envelope of $f$ is 
\begin{equation}\label{eq:poly-envelope}\textstyle
P f(M)
=\inf\{\sum_{i=1}^{d+1}\lambda_if(M_i) : \bp(M)=\sum_{i=1}^{d+1}\lambda_i\bp(M_i)\},
\qquad M\in\R^{m\times n},
\end{equation}
where the infimum is taken over all tuples $(M_1,\ldots,M_{d+1})$ and 
$(\lambda_1,\dots,\lambda_{d+1})$ of $m\times n$ matrices and nonnegative real weights 
that sum to one, respectively---in other words, over all expressions of $M$ as a convex 
combination of $d+1$ matrices that satisfy additional nonlinear constraints (linear in 
the minors) \cite[Theorem~6.8]{dacorogna08}. Because of these additional constraints,
the convex envelope of Example~\ref{exa:convex-env} satisfies $C f\leq P f$. Formula
\eqref{eq:poly-envelope} reads $Pf=g_{d+1}\circ\bp$, where 
\begin{equation}\label{eq:gd1}\textstyle
g_{d+1}(\bq):=
\inf\{ \int f\dd\mu : \text{$\mu\in G_\fint(\bq)$, $\mu$ is $k$-atomic for some $k\leq d+1$}\,\},
\qquad \bq\in\R^d.
\end{equation}
Since $g_\fint\leq g_{d+1}\leq g_\ait$ is clear, we retrieve $Pf=g_{d+1}\circ\bp$ from 
Proposition~\ref{pro:Dacorogna} above. The expression $P f=g\circ\bp$ of 
Corollary~\ref{cor:Dacorogna}, where $g$ is the optimal value function \eqref{eq:PM} on Borel 
measures, is also well known, see e.g.\ \cite{fantuzzi-etal26} and the references therein.
\par
Furthermore, Corollary~\ref{cor:Dacorogna} shows $P f=g_\extt\circ\bp=g_\ait\circ\bp$. Not 
only are the optimal value functions $g_\extt=g_\ait$ equal, but also the constraint sets 
$\ext G(\bq)=G_\ait(\bq)$, $\bq\in\R^d$, by Corollary~\ref{cor:Borel-faces}~(4). To the
best of our knowledge, these assertions are novel, although they follow immediately from 
Weizs\"{a}cker and Winkler's Theorem~\ref{thm:WW-integral} and Dacorogna's formula 
\eqref{eq:poly-envelope}. 
\end{example}
\begin{remark}[Summary]\label{rem:envelopes}
We have shown that optimal value functions of generalized moment problems are generalized 
convex envelopes (Proposition~\ref{pro:Dacorogna} and Corollary~\ref{cor:Dacorogna}). We 
are unaware of literature on this topic aside from the references in 
Example~\ref{exa:convex-env} and Example~\ref{exa:polyconvex-envelope}. Regarding the 
polyconvex envelope, the representations $P f=g_\extt\circ\bp=g_\ait\circ\bp$ using optimal 
value functions on extreme points and on finitely atomic measures that have the affine 
independence property seem to be new. The representations $P f=g_{d+1}\circ\bp=g\circ\bp$ 
with optimal value functions on finitely atomic measures and on Borel measures are well 
known \cite{dacorogna08,fantuzzi-etal26}.
\end{remark}
%
%
\section{Conclusion}
\label{sec:conclusion}
Restricting generalized moment problems to extreme points or to finitely atomic
measures is not just appealing from a geometric perspective. It can also help 
improve algorithms. For example, the polyconvex envelope is pointwise the 
infimum of a generalized moment problem on finitely atomic probability measures on 
$\R^{m\times n}$ (Example~\ref{exa:polyconvex-envelope}). Computing the polyconvex 
envelope is an important, hard problem in material science, for which approximate 
algorithms have been devised \cite{bartels05} that place the atoms on a discrete 
grid in $\R^{m\times n}$. Our results show that the generalized moment problem can 
be restricted to finitely atomic measures that have the affine independence property, 
without increasing the optimal value. This may accelerate the algorithm. A 
different numerical approach to the polyconvex envelope, without discretization, is 
the moment sum-of-squares hierarchy \cite{fantuzzi-etal26}. This is a standard 
technique of polynomial optimization, whose starting point is a generalized moment 
problem on Borel probability measures on $\R^{m\times n}$.
\par
The step from finitely to countably infinitely many moment constraints is substantial
as the Hamburger moment problem shows \cite{schmuedgen17}. If a Borel probability 
measure $\mu$ on $\R$ has compact support and a moment sequence 
$(\int_{-\infty}^\infty x^i\dd\mu(x))_{i\in\N}$ of real numbers, then the constraint 
set is $\{\mu\}$. Hence, extreme points of the constraint sets can be absolutely 
or singular continuous with respect to the Lebesgue measure \cite[pp.~22f]{reed-simon80}.
They are not just countable sums of point measures. It would be interesting to see 
whether the embedding of Theorem~\ref{thm:ext} could give new insights into extreme 
points $\nu$. For $d<\infty$ many moment constraints, the affine map $P^\bp\to\R^d$, 
$\mu\mapsto \int\bp\dd\mu$ restricts to the embedding $F_{P^\bp}(\nu)\to\R^d$, which 
we exploit in Proposition~\ref{pro:ext-fin}~(1) using Minkowski's and  
Carath{\'e}odory's theorems ($P^\bp$ is the set of Borel probability measures with respect 
to which the constraints $\bp$ are integrable). Similarly, there is also an embedding 
$F_{\tilde P}(\nu)\to \R^\N$ where $\tilde P$ is the set of Borel probability measures 
that have a moment sequence. 
\par  
If one is only interested in necessary conditions for points to be extreme points, 
then Proposition~\ref{pro:ext-fin} is also useful for convex sets that are not 
$(d+1)$-neighborly. An example is the set of quantum states 
\cite{alfsen-shultz01,aubrun-szarek17,bengtsson-zyczkowski17}, the positive operators 
of trace one on a separable Hilbert space, which is linearly compact but not 
two-neighborly. Its extreme points are the pure states, being orthogonal projections 
onto one-dimensional subspaces. Considering $d<\infty$ many affine constraints, the 
extreme points are convex combinations of at most $d+1$ pure states having the affine 
independence property. Remarkably, the extreme points of an energy-constrained set of 
states are pure and restricting a lower-bounded lower semicontinuous concave map 
to them preserves the infimum \cite{weis-shirokov21}, because the set of quantum states 
is $\mu$-compact \cite{holevo-shirokov06,protasov-shirokov2009}. 
\par
%
%
\section*{Acknowledgments}
SW thanks Maria Infusino for valuable conversations on moment problems and 
literature on Radon measures \cite{schwartz73}. He is grateful to Heinrich 
v.\ Weizs\"{a}cker for his helpful letter on integral representations of 
Radon measures under moment constraints. This work was funded by the 
European Union under the project ROBOPROX (reg.\ no.\ CZ.02.01.01/00/22\_008/0004590).
\par
%
%

%
%
\end{document}